\documentclass[review]{elsarticle}

\usepackage{lineno,hyperref}
\modulolinenumbers[5]
\usepackage{amsmath, amsthm, amscd, amsfonts,amssymb, graphicx}

 \newtheorem{thm}{Theorem}[section]
 \newtheorem{cor}[thm]{Corollary}
 \newtheorem{lem}[thm]{Lemma}
 \newtheorem{prop}[thm]{Proposition}

 \newtheorem{defn}[thm]{Definition}

 \newtheorem{rem}[thm]{Remark}

\numberwithin{equation}{section}
\usepackage{color}
\def\Hil{\mathcal{H}}

\newcommand{\ltfun}[2]{\sum \limits_{#1} \bigoplus #2}
\def\ltiv{\ltfun{i\in I}{V_i}}
\def\ltiw{\ltfun{i\in I}{W_i}}

\newcommand{\range}[1]{\mathsf{ran}\left( #1 \right)}

\newcommand{\kernel}[1]{\mathsf{ker}\left( #1 \right)}
\newcommand{\finspan}[1]{\mathsf{span}\left( #1 \right)}
\newcommand{\clspan}[1]{\overline{\finspan{#1}}}

\begin{document}

\title{The invertibility of $U$-fusion cross Gram matrices of operators}

\author[author1]{Mitra Shamsabadi}
\ead{mi.shamsabadi@hsu.ac.ir}

\author[author1]{Ali Akbar Arefijamaal}

\author[author2]{Peter Balazs\corref{mycorrespondingauthor}}
\cortext[mycorrespondingauthor]{Corresponding author}
\ead{peter.balazs@oeaw.ac.at}

\address[author1]{Department of Mathematics and Computer Sciences, Hakim Sabzevari University, Sabzevar, Iran}
\address[author2]{Acoustics Research Institute, Austrian Academy of Sciences, Vienna, Austria}

\begin{abstract}
For applications like the numerical solution of physical equations a discretization scheme for operators is necessary.
Recently frames have been used for such an operator representation.
In this paper, we apply fusion frames for this task. We interpret the operator representation using fusion frames as a generalization of fusion Gram matrices. We present the basic definition of  $U$-fusion cross Gram  matrices of operators for a bounded operator $U$. We give sufficient conditions for their (pseudo-)invertibility and present explicit formulas for the inverse.
In particular, we characterize  fusion Riesz bases and fusion orthonormal bases by such matrices.
Finally, we look at which perturbations of fusion Bessel sequences preserve the invertibility of the  fusion Gram  matrix of operators.
\end{abstract}

\maketitle

\section{Introduction and motivation}
For the representation (and modification) of functions a standard approach is using orthonormal bases (ONBs).
It can be hard to find a 'good' orthonormal basis, in the sense that it sometimes cannot fulfill given   properties, as formally expressed e.g. in the Balian-Low theorem \cite{gr01}.
For solving this problem, frames were introduced by Duffin and Schaeffer \cite{duffschaef1} and widely developed by many authors \cite{befe01,ch08,daubgromay86,feistro1}.
In recent years, frames have been the focus of active research, both in theory \cite{ALDROUBI20171121,
BODMANN2010397,FUHR2019563} and applications \cite{framepsycho16,boelc1,colgaz10}.
Also, several generalizations have been investigated, e.g. \cite{ALDROUBI20081667,antoin2,jpaxxl09,spexxl14,Sun2006437}, among them fusion frames \cite{caskut04,cakuli08,ga07}, which are the topic of this paper.

For a numerical treatment of operator equations, used for example
for solving integral equations in acoustics \cite{Kreuzeretal09},
the involved operators have to be discretized to be handled
numerically. The (Petrov-)Galerkin approach \cite{Gauletal03} is a
particular and well-known way for this discretization. For
an operator $O,$ the matrix $M$
defined by $M_{k,l} = \left< O \psi_{l}, \phi_{k}
\right>$ is called the matrix corresponding to the
operator $O$, or the system matrix. The standard way for the
discretization of operators is using bases \cite{gohberg1}, but
recently the general theory for frames has been developed
\cite{xxlframoper1,xxlgro14}. Frames were also used in numerics
\cite{harbr08}, in particular in an adaptive approach
\cite{Dahlkeetal07c,Stevenson03}. In \cite{xxlriek11,xxlrieck11}
sufficient and necessary conditions of the invertibility of such
matrices is investigated.
Note that the system matrix of the identity is the cross Gram matrix of the two sequences $\{\psi_k\}_{k\in I}$ and $\{\phi_k\}_{k\in I}$. Therefore, in \cite{mitraG} the concept of matrix representation of operators using frames is reinterpreted as a generalization of the Gram matrix to investigate the inverses.
As the concept of domain decomposition is a particularly relevant topic in this field, the extension of the approach to operator representations to fusion frame is very useful \cite{oswald09}.

In this paper, we therefore look at those $U$-fusion cross Gram matrices. In particular, we investigate the (pseudo-)invertibility of $U$-fusion Gram matrices of operators. In Section \ref{sec:prenot0}, we review basic notations and preliminaries. In Section \ref{sec:fusionGram0}, we give necessary as well as sufficient conditions for the {(pseudo-)}invertibility of the $U$-cross Gram matrices, characterize  fusion orthonormal bases and fusion Riesz bases by those properties and give formulas for the (pseudo-)inverses.
Finally, in Section \ref{sec:stabgram0}, some stability results are discussed.

\section{Preliminaries and Notations}\label{sec:prenot0}

Throughout this paper, $\mathcal{H}$ is a separable Hilbert space,
$I$ a countable index set and $I_{\mathcal{H}}$ the identity
operator on $\mathcal{H}$ and $\{e_{i}\}_{i\in I}$  an orthonormal basis for $\mathcal{H}$.
The orthogonal projection on a subspace $V\subseteq\mathcal{H}$ is denoted by $\pi_V$.
We will denote the set of all linear and bounded operators between Hilbert spaces
$\mathcal{H}_{1}$ and $\mathcal{H}_{2}$ by
$B(\mathcal{H}_{1},\mathcal{H}_{2})$ and for
$\mathcal{H}_{1}=\mathcal{H}_{2}=\mathcal{H}$, it is represented
by $B(\mathcal{H})$.
 We denote the range and the null space
of an operator $U$ by $\range{U}$ and $\kernel{U}$, respectively.
For a closed range operator  $U\in B(\Hil_1,\Hil_2)$, the pseudo-inverse of $U$ is defined  the unique
operator $U^{\dag}\in B(\Hil_2,\Hil_1)$ satisfying
\begin{eqnarray*}
\kernel{U^{\dag}}=\range{U}^{\perp},\
\range{U^{\dag}}=\kernel{U}^{\perp},\
\text{and}\
UU^{\dag}U=U.
\end{eqnarray*}
 The operator $U$ has closed range  if and only if $U^*$ has
closed range and
$\left(U^*\right)^{\dag}=\left(U^{\dag}\right)^{*}$, see e.g.
 \cite[ Lemma 2.4.1, Lemma 2.5.2]{ch08}.

 For $0<p<\infty$, the space of all compact operators (which are the closure of finite-rank operators and denoted by $K(\Hil)$) $T$ on
$\mathcal{H}$ such that its singular values $\{\lambda_{n}\}_{n\in I}$ belonging to $\ell^{p}$ is called the Schatten $p$-class of $\mathcal{H}$. It is denoted by $S_{p}(\mathcal{H})$ which is a Banach space with the norm
\begin{eqnarray}\label{norm}
\|T\|_{p}=\left(\sum_{n}|\lambda_{n}|^{p}\right)^{\frac{1}{p}}.
\end{eqnarray}
An operator $T\in B(\mathcal{H})$ is called  \emph{trace class} if $trace(T) :=
\sum_{i\in I}\left\langle Te_i,e_i\right\rangle<\infty,$ for every orthonormal basis $\{e_i\}_{i\in I}$ for $\Hil$.
It is shown that $T$ is trace class if and only if $T\in S_{1}(\Hil)$. Also, the class of \emph{Hilbert-Schmidt operators} of $\mathcal{H}$ is denoted by  $S_{2}(\mathcal{H})$,  and $T\in S_{2}(\Hil)$ if and only if $\|T\|^2_2=\sum_{i\in I}\left\| Te_i\right\|^2<\infty$.
It is well-known that $K(\Hil)$ and $S_p(\mathcal{H})$ are two sided $*$-ideal of
$B(\mathcal{H})$, that is, a Banach algebra under the norm
(\ref{norm}) and the finite rank operators are dense in
$(S_p(\mathcal{H}), \|.\|_p)$. Moreover, for $T\in
S_p(\mathcal{H})$, one has $\|T\|_p = \|T^*\|_p$ and $ \|T\|\leq\|T\|_p$.
If $S_1\in B(\mathcal{H},\mathcal{H}_1)$ and $S_2\in B(\mathcal{H}_2,\mathcal{H})$, then $\|S_1 T\|_p
\leq\|S_1\|\|T\|_p$ and $\|T S_2\|_p\leq\|S_2\|\|T\|_p$.
For more information about these operators, see
\cite{gohberg1,pitch,schatt1,weidm80}.

\subsection{Fusion frames}

We now review some definitions and primary results
 of fusion frames. For more information see \cite{caskut04,cakuli08,ga07}.

For each sequence $\{W_{i}\}_{i\in I}$ of closed subspaces in $\mathcal{H}$, the space
\begin{equation*}
\left(\sum_{i\in I}\bigoplus W_{i}\right)_{\ell^{2}}=\left\{\{f_{i}\}_{i\in I}: f_{i}\in W_{i}, \sum_{i\in I}\|f_{i}\|^{2}<\infty \right\},
\end{equation*}
with the inner product
\begin{equation*}
\bigg\langle\{f_{i}\}_{i\in I},\{g_{i}\}_{i\in I} \bigg\rangle=\sum_{i\in I}\langle f_{i},g_{i}\rangle,
\end{equation*}
is a Hilbert space.

We now give the central definition of fusion frames:
\begin{defn}
Let $\{W_{i}\}_{i\in I}$ be a family of closed subspaces of $\mathcal{H}$ and $\{\omega _{i}\}_{i\in I}$
be a family of weights, i.e. $\omega_{i}>0$, $i\in I$. The sequence  $\{(W_{i},\omega_{i}
)\}_{i\in I}$
is called a \emph{fusion frame} for $\mathcal{H}$ if there exist constants $0<A_{W}\leq B_{W}<\infty$ such that
\begin{equation*}
A_{W}\|f\|^{2}\leq\sum_{i\in I}\omega _{i}^{2}\|\pi_{W_{i}}f\|^{2}\leq B_{W}\|f\|^{2}, \qquad (f\in \mathcal{H}).
\end{equation*}
The constants $A_{W}$ and $B_{W}$ are called \emph{fusion frame
bounds}. If we  have the upper bound, we call
$\{(W_{i},\omega_{i})\}_{i\in I}$ a $\emph{Bessel fusion
sequence}$. A fusion frame is called $\emph{tight}$, if $A_{W}$
and $B_{W}$ can be chosen to be equal, and $\emph{Parseval}$ if
$A_{W}=B_{W}=1$. If $\omega_{i}=\omega$ for all $i\in I$, the
collection $\{(W_{i},\omega_{i})\}_{i\in I}$ is called
$\emph{$\omega$-uniform}$.
A fusion frame
$\{(W_{i},\omega_{i} )\}_{i\in I}$ is said to be a $
\emph{fusion orthonormal basis}$ if $\mathcal{H}=\bigoplus_{i\in
I}W_{i}$ and it is called  a $\emph{Riesz decomposition}$ of
$\mathcal{H}$ if for every $f\in \mathcal {H} $ there is a unique
choice of $f_{i}\in W_{i}$ such that $f=\sum_{i\in I}f_{i}$. A
family of subspaces is called
 \em{complete} if $\clspan{W_i} = \Hil$.
\end{defn}
It is clear that every fusion orthonormal basis is a Riesz decomposition of $\mathcal{H}$.
Moreover, a family $\{W_{i}\}_{i\in I}$ of closed subspaces of $\mathcal{H}$
 is a fusion orthonormal basis if and only if it is
 a 1-uniform Parseval fusion frame \cite{caskut04}.\\

Furthermore, the \emph{synthesis operator}
 $T_{W}:(\sum_{i\in I}\bigoplus W_{i})_{\ell^{2}}\rightarrow \mathcal{H}$
for a Bessel fusion sequence $ W=\{(W_{i},\omega_{i})\}_{i\in I}$ is
defined by
\begin{equation*} \label{eq:sybthop}
T_{W}(\{f_{i}\}_{i\in I})=\sum_{i\in I}\omega_{i}f_{i}.
\end{equation*}
The adjoint operator $T^{*}_{W}: \mathcal{H}\rightarrow (\sum_{i\in I}\bigoplus W_{i})_{\ell^{2}}$ which is called the $\emph{analysis operator}$
is given by
\begin{equation*}
T^{*}_{W}f=\{\omega_{i}\pi _{W_{i}}f\}_{i\in I}, \qquad  (f\in \mathcal{H}).
\end{equation*}
Both are bounded by $\sqrt{B_W}$.

If $ W=\{(W_{i},\omega_{i}
)\}_{i\in I}$ is a fusion frame, the $\emph{fusion frame operator}$ $S_{W}:\mathcal{H}\rightarrow \mathcal{H}$, which is defined by $S_{W}f=T_{W}T^{*}_{W}f=\sum_{i\in I}\omega_{i}^{2}\pi _{W_{i}}f$,
 is  bounded (with bound $B_W$), invertible and positive \cite{caskut04,ga07}.

Every Bessel fusion sequence
 $V=\{(V_{i},\upsilon_{i})\}_{i\in I}$ is called a $\emph{G\v{a}vru\c{t}a-dual}$
  of  $W=\{(W_{i},\omega_{i})\}_{i\in I}$, if
\begin{equation*}\label{sec:reconstr1}
f=\sum_{i\in I}\omega_{i}\upsilon_{i}\pi_{V_{i}}S_{W}^{-1}\pi_{W_{i}}f,
\qquad (f\in \mathcal{H}),
\end{equation*}
 for more details see \cite{ga07}. From here on, for simplicity we say dual instead of  G\v{a}vru\c{t}a-dual.  The sequence of subspaces $\widetilde{W} := \left\{\left(S_{W}^{-1}W_{i},\omega_{i}\right)\right\}_{i\in I}$ is also a fusion frame for $\mathcal{H}$ and a dual of $W,$ called the \emph{canonical dual} of $W$ \cite{caskut04,ga07}.
Rephrasing that, a Bessel fusion sequence $V=\{(V_{i},\upsilon_{i})\}_{i\in I}$ is a dual of a
fusion frame $ W=\{(W_{i},\omega_{i})\}_{i\in I}$ if and only if
\begin{equation} \label{sec:duality}
T_{V}\phi_{VW}T^{*}_{W}=I_{\mathcal{H}} ,
\end{equation}
where the bounded operator $\phi_{VW}:(\sum_{i\in I}\bigoplus
W_{i})_{\ell^{2}}\rightarrow (\sum_{i\in I}\bigoplus
V_{i})_{\ell^{2}}$ is given by
\begin{equation}\label{phi}
\phi_{VW}(\{f_{i}\}_{i\in I})=\{\pi_{V_{i}}S_{W}^{-1}f_{i}\}_{i\in I}
\end{equation}
  and $\left\|\phi_{VW}\right\|\leq\left\|S_{W}^{-1}\right\|$. Also, $V$ is called a \textit{pseudo-dual} of $W$ if $T_{V}\phi_{VW}T^{*}_{W}$ is an invertble operator, see \cite{ChAppro10} for discrete case.
Another approach to duality \cite{hemo14,behemoza14} uses any bounded operator $O:(\ltiw)_{\ell^2} \rightarrow (\ltiv)_{\ell^2}$. Starting with two fusion frames the duality is defined analogously to \eqref{sec:duality}, i.e.
$T_{V} \, O \, T^{*}_{W}=I_{\mathcal{H}}$.
 We stick to the G\v{a}vru\c{t}a duals, but all results herein can be adapted to this other definition of duality.

 Let $\{W_{i}\}_{i\in I}$ be a family of closed subspaces of $\mathcal{H}$ and $\{\omega_{i}\}_{i\in I}$  a family of weights. We say that $\{(W_{i}, \omega_{i})\}_{i\in I}$ is a \emph{fusion Riesz basis} for $\mathcal{H}$ if
 $\clspan{W_i}=\mathcal{H}$ and there exist constants $0<C\leq D<\infty$ such that for each finite subset $J\subseteq I$   and all $(f_{j}\in W_{j}, j\in J)$ we have
\begin{equation}\label{defriesz}
C\sum_{j\in J}\|f_{j}\|^{2}\leq \left\|\sum_{j\in J}\omega_{j}
f_{j}\right\|^2\leq D\sum_{j\in J}\|f_{j}\|^{2}.
\end{equation}

\begin{rem} \label{rem:Riszdec}  Note that weights are not included in the definition of the Riesz decomposition. Obviously, we have that $f=\sum_{i\in I}f_{i}$ with unique $f_i$s, if and only if  $f=\sum_{i\in I}w_i f_{i}$ with the same uniqueness.
\end{rem}

The next theorem explores fusion Riesz bases with respect to local frames and their operators.
\begin{thm}\cite{caskut04}\label{4}
Let $\{W_{i}\}_{i\in I}$ be a fusion frame for $\mathcal{H}$ and $\{e_{ij}\}_{j\in J_{i}}$ be an orthonormal basis for $W_{i}$, for each $i\in I$. Then the following conditions are equivalent:
\begin{enumerate}
\item[(1)] $\{W_{i}\}_{i\in I}$ is a Riesz decomposition of $\mathcal{H}$.
\item[(2)] The synthesis operator $T_{W}$ is one to one.
\item[(3)] The analysis operator $T^{*}_{W}$ is onto.
\item[(4)] $\{W_{i}\}_{i\in I}$ is a fusion Riesz basis for $\mathcal{H}$.
\item[(5)] $\{e_{ij}\}_{i\in I, j\in J_{i}}$ is a Riesz basis for $\mathcal{H}$.
\end{enumerate}
\end{thm}

For some results we need  a version of Theorem \ref{4} formulated for general family of subspaces:
\begin{prop} \label{4new}
Let $W=\{(W_{i},\omega_i)\}_{i\in I}$ be a family of closed subspaces and $\{f_{ij}\}_{j\in J_{i}}$ be a Riesz basis for $W_{i}$, for each $i\in I$ with bounds $A_i$ and $B_i$, respectively, such that
\begin{eqnarray*}
0<\inf _{i\in I}A_i\leq \sup_{i\in I}B_i<\infty.
\end{eqnarray*}
 Then the following conditions are equivalent:
\begin{enumerate}
\item[(1)] $W$ is a Riesz decomposition of $\mathcal{H}$.
\item[(2)] The synthesis operator $T_{W}$ is bounded and bijective.
\item[(3)] The analysis operator $T^{*}_{W}$ is bounded and bijective.
\item[(4)] $W$ is a fusion Riesz basis for $\mathcal{H}$.
\item[(5)] $\{\omega_if_{ij}\}_{i\in I, j\in J_{i}}$ is a Riesz basis for $\mathcal{H}$.
\end{enumerate}
\end{prop}
\begin{proof} $(1)$ $\Leftrightarrow$ $(2)$ $\Leftrightarrow$ $(3)$ by Theorem \ref{4} for any family of closed subspaces (as those conditions imply a fusion frame property).

$(4)\Leftrightarrow (5)$ by relating the inequality (\ref{defriesz}) for fusion Riesz basis $W$  to the one for Riesz bases $\{w_{i}f_{ij}\}_{i\in I,j\in J_i}$ see e.g. \cite[Theorem 3.3.7]{ch08}.

(4) $\Leftrightarrow$ (2)
 The equation \eqref{defriesz} is equivalent to $T_W$ being injective, having closed range  and being bounded. By the completeness we have $(4) \Rightarrow (2)$. By the surjectivity of $T_W$ we have completeness and so $(2) \Rightarrow (4)$.

\end{proof}

The following characterizations of fusion Riesz bases will be used frequently in this note,  which is a generalization of \cite[Theorem 2.2]{mitra} to non-uniform fusion frames.

\begin{prop}\label{riesz mitra}
 Let $W=\{(W_{i},\omega_i)\}_{i\in I}$ be a fusion frame in $\mathcal{H}$. Then the following are equivalent:
\begin{enumerate}
\item[(1)] $W$ is a fusion Riesz  basis.
\item[(2)] $S_{W}^{-1}W_{i}\perp W_{j}$ for all $i, j \in I, i\neq j$.
\item[(3)] $\omega_{i}^2\pi_{W_{i}}S_{W}^{-1}\pi_{W_{j}}=\delta_{ij}\pi_{W_{j}}$ for all $i, j \in I$.
\end{enumerate}
 \end{prop}
\begin{proof}
$(1)\Rightarrow (2)$ was proved in \cite[Proposition 4.3]{caskut04}. \\
$(2)\Rightarrow (1)$ Suppose that $\{e_{ij}\}_{j\in J_{i}}$ is
 an orthonormal basis for $W_{i}$, for all $i\in I$. Then
 for any $f\in \mathcal{H}$ we have
 \begin{eqnarray*}
 \sum_{i\in I}\sum_{j\in J_{i}}\left|\left<f,\omega_ie_{ij}\right>\right|^{2}=\sum_{i\in I}\omega_i^2\left\|
 \pi_{W_{i}}f\right\|^{2}.
 \end{eqnarray*}
  It easily follows
  that $\{e_{ij}\}_{i\in I, j\in J_{i}}$ is a weighted frame \cite{xxljpa1} (with weights $\omega_i>0$) for $\mathcal{H}$ with the frame operator $S_{W}$.
  Moreover, by Proposition \ref{4new} it is enough to show that
 $\{\omega_ie_{ij}\}_{i\in I, j\in J_{i}}$
is a Riesz basis for $\mathcal{H}$ or equivalently, that the sequences $\{\omega_ie_{ij}\}_{i\in I,j\in J_{i}}$ and $\{S_{W}^{-1}\omega_ie_{ij}\}_{i\in I,j\in J_{i}}$ are biorthogonal.
This immediately follows from the reconstruction formula
\begin{eqnarray*}
f=\sum_{i\in I}\sum_{j\in J_{i}}\langle f,S_{W}^{-1}\omega_ie_{ij}\rangle \omega_ie_{ij}, \qquad (f\in \mathcal{H}).
\end{eqnarray*}

$(2)\Rightarrow (3)$
Suppose that $\{(W_{i},\omega_i)\}_{i\in I}$ is a fusion frame in $\mathcal{H}$ and
 $f,g\in \mathcal{H}$. By (2) we obtain
\begin{eqnarray*}
\langle\pi _{W_{i}}S_{W}^{-1}\pi _{W_{j}}f,g\rangle=
\langle S_{W}^{-1}\pi_{W_{j}}f,\pi_{W_{i}}g\rangle=0 \qquad  (i\neq j).
\end{eqnarray*}
In particular, suppose $\{e_{ij}\}_{j\in J_{i}}$ is an orthonormal
basis for $W_{i}$, then we know by the argument  above that the sequence $\{\omega_ie_{ij}\}_{i\in I, j\in J_{i}}$
is a Riesz basis for $\mathcal{H}$ with the frame operator $S_W$. Hence, $\{S_{W}^{-1/2}\omega_ie_{ij}\}_{i\in I ,j\in {J_{i}}}$ is an
orthonormal basis for $\mathcal{H}$.
Let $f,g\in \mathcal{H}$ and take
\begin{eqnarray*}
\pi_{W_{i}}f=\sum_{j\in J_{i}}c_{ij}e_{ij},\ \ \ \pi_{W_{i}}g=\sum_{j\in J_{i}}d_{ij}e_{ij},
\end{eqnarray*}
for all $i\in I$, for some $\{c_{ij}\}_{j\in J_{i}}$ and $\{d_{ij}\}_{j\in J_{i}}$ in $\ell^{2}$. Then, for all $i\in I$ we have
\begin{eqnarray*}
\left\langle \omega_i^2\pi_{W_{i}}S_{W}^{-1}\pi_{W_{i}}f,g\right\rangle &=&
\left\langle S_{W}^{-1/2}\omega_i\pi_{W_{i}}f,S_{W}^{-1/2}\omega_i\pi_{W_{i}}g\right\rangle\\
&=&\langle\sum_{j\in J_{i}}c_{ij}S_{W}^{-1/2}\omega_ie_{ij},\sum_{k\in J_{i}}d_{ik}S_{W}^{-1/2}\omega_ie_{ik}\rangle\\
&=&\sum_{j,k\in J_{i}}c_{ij}\overline{d_{ik}}\left\langle S_{W}^{-1/2}\omega_ie_{ij},S_{W}^{-1/2}\omega_ie_{ik}\right\rangle\\
&=&\left\langle\sum_{j\in J_{i}}c_{ij}e_{ij},\sum_{k\in J_{i}}d_{ik}e_{ik}\right\rangle=\left\langle\pi_{W_{i}}f,g\right\rangle.
\end{eqnarray*}
So, $\omega_i^2\pi_{W_{i}}S_{W}^{-1}\pi_{W_{i}}=\pi_{W_{i}}$.

$(3)\Rightarrow (2)$ Let $f\in W_{i}$ and $g\in W_{j}$, where
$i\neq j$. Then
\begin{eqnarray*}
\langle S_{W}^{-1}f,g\rangle
&=&\langle S_{W}^{-1}\pi_{W_{i}}f, \pi_{W_{j}}g\rangle\\
&=&\langle \pi_{W_{j}}S_{W}^{-1}\pi_{W_{i}}f,g\rangle=0.
\end{eqnarray*}
\end{proof}
From the Proposition \ref{riesz mitra}, it easily follows that for a fusion Riesz basis $W$
\begin{equation}\label{**}
\phi_{WW}(f_i)=\{\pi_{W_i}S_{W}^{-1}f_i\}_{i\in I}=\{\pi_{W_i}S_{W}^{-1}\pi_{W_i}f_i\}_{i\in I}=\{\frac{1}{\omega_i^2}f_i\}_{i\in I}.
\end{equation}
Therefore, it is bounded and the operator $f_{i} \mapsto w_{i}^2 f_{i}$ is its bounded inverse.

All items in Proposition  \ref{riesz mitra} include a dependency on weights, for (2) the weight is included in the definition of $S_W$. So, the question arises how dependent the Riesz property is on the considered weights. It is worthwhile to mention that, by the fusion frame definition, the family of weights belongs to $\ell_{+}^{\infty}$  assuming that the subspaces are non-empty, see \cite[Remark 2.4]{behemoza14}.
Moreover, if $W=\{(W_i,w_i)\}_{i\in I}$ is a fusion Riesz basis, then  (\ref{defriesz}) shows that
\begin{eqnarray*}
\sqrt{C}\leq w_i\leq \sqrt{D}, \qquad(i\in I).
\end{eqnarray*}
Using Proposition \ref{4new} the following lemma immediately follows.
\begin{lem} \label{cor:Riestweight0} Let $W = \{(W_{i},\omega_i)\}_{i\in I}$  and $V = \{(W_{i},v_i)\}_{i\in I}$ be a family of subspaces in $\Hil$ with different weights. Then $W$  is a  fusion Riesz basis if and only if $V$ is a fusion Riesz basis.\\
\end{lem}

This could also be seen as direct consequence of  Proposition \ref{4new}, as the Riesz decomposition property, e.g. item $(1)$ is independent of the weight.

In the sequel, for a given fusion Riesz basis $W=\{(W_{i}, w_{i})\}_{i \in I}$, we denote by $W'$ the $1$-uniform family of subspaces $\{(W_{i}, 1)\}_{i \in I}$.

 We will use the following criterion for the invertibility
of operators.
\begin{prop} \cite{gohberg1}\label{invOP}
Let $F:\mathcal{H}\rightarrow \mathcal{H}$ be invertible on $\mathcal{H}$.
Suppose that $G:\mathcal{H}\rightarrow \mathcal{H}$ is a bounded
 operator and $\|Gh-Fh\|\leq \upsilon\|h\|$ for all $h\in \mathcal{H}$,
  where $\upsilon\in[0,\frac{1}{\|F^{-1}\|})
$. Then
 $G$ is invertible on $\mathcal{H}$ and
    $G^{-1}=\sum_{k=0}^{\infty}[F^{-1}(F-G)]^{k}F^{-1}$. \end{prop}

\section{$U$-fusion cross Gram matrix of operators}\label{sec:fusionGram0}
In this section, we extend the notion of cross Gram matrices \cite{mitraG} to fusion frames and discuss on their invertibility.

We interpret the representation of operators using fusion
frames \cite{xxlcharshaare18} as a generalization of the Gram matrix of operators:

\begin{defn}
Let $W=\{(W_{i},\omega_{i})\}_{i\in I}$ be a Bessel fusion sequence for $\mathcal{H}$ and $V=\{(V_{i},\upsilon_{i})\}_{i\in I}$ a fusion frame for $\mathcal{H}$. For $U\in B(\mathcal{H})$, the matrix operator
$\mathcal{G}_{U,W,V}:\left(\sum_{i\in I}\bigoplus W_{i}\right)_{\ell^{2}}\rightarrow \left(\sum_{i\in I}\bigoplus W_{i}\right)_{\ell^{2}}$ given by
\begin{eqnarray*}
\mathcal{G}_{U,W,V}=\phi_{WV}T_{V}^{*}UT_{W},
\end{eqnarray*}
 is  called the \emph{$U$-fusion cross Gram matrix}. If  $U=I_{\mathcal{H}}$, it is called \emph{fusion cross Gram matrix} and denoted by $\mathcal{G}_{W,V}$. We use $\mathcal{G}_{W}$ for $\mathcal{G}_{W,W}$; the so called \emph{fusion Gram matrix}.
\end{defn}

Note that
$$\left[ \mathcal{G}_{U,W,V} (f_i) \right]_j = \left[ \phi_{WV}\left\{v_k \pi_{V_k} U \sum \limits_i w_i f_i \right\}_{k\in I}\right]_{j\in I} = \sum \limits_i \underbrace{  w_iv_j \pi_{W_j}S_V^{-1}\pi_{V_j} U}_{B_{j,i}} f_i = \sum \limits_i B_{j,i} f_i, $$
where $B_{j,i} : W_i \rightarrow W_j$. Therefore $\mathcal{G}_{U,W,V}$ is a block-matrix of operators \cite{MANA:MANA19941670102}\footnote{This could be called a generalized subband matrix, motivated by system identification applications \cite{PerfAnSBI}.},  which motivates the name (cross-)Gram  matrix.

Clearly, using \eqref{phi},
$U$-fusion cross Gram matrices are well-defined and
\begin{eqnarray*}
\left\|\mathcal{G}_{U,W,V}\right\|
&=&\left\|\phi_{WV}T_{V}^{*}UT_{W}\right\|\\
&\leq&\left\|\phi_{WV}\right\|\left\|T_{V}^{*}\right\|
\|U\|\left\|T_{W}\right\|\\
&\leq& \left\|S_{V}^{-1}\right\|\|U\|\sqrt{B_{W}B_{V}}\\
&\leq&  \frac{\sqrt{B_W  B_V}}{A_V} \|U\|.
\end{eqnarray*}

 We have chosen to use $\mathcal{G}_{U,W,V}=\phi_{WV}T_{V}^{*}UT_{W} $ instead of the 'naive'
 $\mathfrak{G}_{U,W,V}= T_{V}^{*}UT_{W}$. The reason for that is that, by this definition, $\mathcal{G}_{U,W,V}$ maps $\sum \limits_{i \in I}\bigoplus {W}_i$ into itself  and is a projection as in the Hilbert space case (albeit an oblique one, see Remark \ref{sec:gengram1}).

We can represent an operator $U\in B(\mathcal{H})$ from its $U$-fusion cross Gram matrix. Suppose $W$ is a dual fusion frame of $V$, then
\begin{equation}\label{U from Gram}
T_W\mathcal{G}_{U,W,V}T_W^*S_W^{-1}=T_W\phi_{WV}T_{V}^{*}UT_{W}T_W^*S_W^{-1}=U.
\end{equation}
From the ideal property of $S_p(\mathcal{H})$ in $B(\mathcal{H})$ it follows that if $U$ is compact, trace class  and Hilbert Schmidt, so is $\mathcal{G}_{U,W,V}$. By using (\ref{U from Gram}) we can deduce the converse with some assumptions.
  Moreover,
  \begin{eqnarray*}
  \phi_{VW} \; \mathcal{G}_{U,W,V}^*=\mathcal{G}_{U^*,V,W} \; \phi_{WV}^*.
  \end{eqnarray*}

\begin{rem} \label{sec:gengram1}
Note that as in the discrete Hilbert space frame case $\mathcal{G}_{V,W} =
\mathcal{G}_{V,W}^2$ and so this an oblique
projection whenever $V$ is a dual fusion frame of $W$. Also $T_{V} \mathcal{G}_{V,W}
= T_{V}$, and $\mathcal{G}_{V,W} \left(
\phi_{V W}  T_{W}^{*} \right)  = \phi_{V
W} T_{W}^*$, but in general the operator is not self-adjoint.
 Using the above identities and the definition of
$\mathcal{G}_{V,W}$ we achieve
\begin{eqnarray*}
\kernel{\mathcal{G}_{V,W}}=\kernel{T_V}
\end{eqnarray*} and
\begin{eqnarray*}
\range{\mathcal{G}_{V,W}}=\left\{
\{\pi_{V_i}S_{W}^{-1}f_{i}\}_{i\in I}:\quad \{f_{i}\}_{i\in I}\in
\range{T^{*}_W}\right\}.
\end{eqnarray*}
\end{rem}
In the next result, we are going to characterize Gram
matrices of fusion orthonormal bases.
\begin{prop}\label{ort}
Let $W=\left\{(W_{i},\omega_i)\right\}_{i\in I}$ be a fusion Riesz basis. The following are equivalent.
\begin{enumerate}
\item[(1)] $W$ is a fusion orthonormal basis.
\item[(2)] $\mathcal{G}_{W}=I_{\sum_{i\in I}\bigoplus W_i}$.
\item[(3)] $\mathcal{G}_{W'}=I_{\sum_{i\in I}\bigoplus {W'}_i}$.
\end{enumerate}
\end{prop}
\begin{proof}
$(1)\Rightarrow (2)$
Assume that $W$ is a fusion orthonormal basis. Applying (\ref{**}) for all ${\bf f}=\{f_i\}_{i\in I}\in \ltiw$ we have
\begin{eqnarray*}
\mathcal{G}_{W}\mathbf{f}&=&\phi_{WW}T_{W}^{*}T_{W}\mathbf{f}\\
&=&\phi_{WW}\left\{w_i\pi_{W_{i}}\sum_{j\in I}w_jf_{j}\right\}_{i\in I}\\
&=&\left\{\frac{1}{w_i^2}w_i\pi_{W_{i}}\sum_{j\in I}w_jf_{j}\right\}_{i\in I}\\
&=&\left\{\pi_{W_{i}}f_{i}\right\}_{i\in I}=\mathbf{f}.
\end{eqnarray*}

$(2)\Rightarrow(1)$ By Proposition \ref{riesz mitra} we obtain
\begin{eqnarray*}
S_{W'}&=& S_{W'}T_WT_W^{-1}\\
&=&\sum_{i\in I}\pi_{W_i}T_WT_W^{-1}\\
&=&\sum_{i\in I} w_i^2\pi_{W_i}S_W^{-1}\pi_{W_i}T_WT_W^{-1}\\
&=&T_W\phi_{WW}T_W^*T_WT_W^{-1}\\
&=&T_W\mathcal{G}_{W}T_W^{-1}=I_{\Hil}.
\end{eqnarray*}

$(1)\Rightarrow (3)$
It follows from $(1)\Rightarrow (2)$ and the fact that $W$ is  a fusion orthonormal basis if and only if $W'$ is  a fusion orthonormal basis.

$(3)\Rightarrow (1)$ It is enough to show that $W'$ is  a Parseval fusion frame by Proposition 3.23 of \cite{caskut04}.
\begin{eqnarray*}
S_{W'}^2&=&T_{W'}T_{W'}^*T_{W'}T_{W'}^*\\
&=&T_{W'}\mathcal{G}_{W'}T_{W'}^*=S_{W'}.
\end{eqnarray*}
Now, the invertibility of $S_{W'}$  implies $S_{W'}=I_{\Hil}.$

\end{proof}
\subsection {Invertibility of $U$-cross Gram matrices}
We now discuss the relationship between the invertibility of Gram matrices and their associated operators.
\begin{prop}\label{INV}
Let $W=\{(W_{i}, \omega_i)\}_{i\in I}$ be a  fusion frame in $\mathcal{H}$ and $U\in B(\mathcal{H})$. The following are equivalent:
\begin{enumerate}
\item[(1)] $W$ is a fusion  Riesz basis and $U$ is invertible.
\item[(2)] $\mathcal{G}_{U,W,W}$ is invertible.
\item[(3)] $\mathcal{G}_{U,W,W}$ is onto.
\item[(4)] $\mathcal{G}_{U,W,W}$ is one to one.
\item[(5)] $\mathcal{G}_{U,\widetilde{W},W}$ is invertible.
\item[(6)] $\mathcal{G}_{U,\widetilde{W},W}$ is onto.
\item[(7)] $\mathcal{G}_{U,\widetilde{W},W}$ is one to one.
\end{enumerate}
\end{prop}

\begin{proof}
 $(1)\Rightarrow (2)$ is trivial by the invertibility of $\phi_{WW}$, see \eqref{**}.

$(2)\Rightarrow (1)$  By using the invertibility
$\mathcal{G}_{U,W,W}=\phi_{WW}T_{W}^*UT_{W}$, the operator $T_{W}$
is injective and so $W$ is a fusion Riesz basis by Theorem
\ref{4}. So, $\phi_{WW}$ is invertible and then
\begin{eqnarray*}
U=\left(T_W^*\right)^{-1}\phi_{WW}^{-1}\mathcal{G}_{U,W,W}T_W^{-1}
\end{eqnarray*}
is invertible.

 $(2)\Rightarrow (3)$ and $(2)\Rightarrow (4)$ are trivial.

    $(3)\Rightarrow (1)$ If $\mathcal{G}_{U,W,W}=\phi_{WW}T_W^* UT_W$ is onto, then  $\phi_{WW}^*=\phi_{WW}$ is invertible. Hence, $T_W^*UT_W=\phi_{WW}^{-1}\mathcal{G}_{U,W,W}$ is onto. Thus, $W$ is a fusion Riesz basis and $U$ is invertible.

$(1)\Rightarrow (5)$
Using  Theorem 2.9 of \cite{Arab17}, $\widetilde{W}$ is also a fusion Riesz basis
and so by (\ref{**}), $T_{\widetilde{W}}$
is invertible by Theorem \ref{4}.
 Hence, $\mathcal{G}_{U,\widetilde{W},W}=\phi_{\widetilde{W}W}T_W^*UT_{\widetilde{W}}$ is invertible by the invertibility of  $U$ and $\phi_{\widetilde{W}W}\{f_i\}_{i\in I}=\left\{S_W^{-1}f_i\right\}_{i\in I}$, for all $\{f_i\}_{i\in I}\in \sum_{i\in I}\bigoplus W_i$.

 $(5)\Rightarrow (1)$
Note that $\phi_{\widetilde{W}W}$ is invertible by  the definition of $\phi_{\widetilde{W}W}$. Also, $T_{\widetilde{W}}$ is one to one since $\mathcal{G}_{U,\widetilde{W},W}=\phi_{\widetilde{W}W}T_{W}^*UT_{\widetilde{W}}$ is invertible and therefore $\widetilde{W}$ is a fusion Riesz basis by Theorem \ref{4}.  By Theorem 2.9  of \cite{Arab17},  $W$ is also a fusion Riesz basis.
  Applying Theorem \ref{4}, $T_W^*$ and $T_{\widetilde{W}}$ are invertible and it immediately follows the desired result.

 $(1)\Leftrightarrow (4) \Leftrightarrow (6)\Leftrightarrow (7)$ can be proved in an analogue way.

\end{proof}

Similar to the question, when the inverse of a frame multiplier is a multiplier again \cite{xxlmult1,stobalrep11}, we can show that for fusion Riesz bases the operator keeps its structure.  To state our results in a accessible way, we need a new definition:
\begin{defn}
Let  $W=\{(W_i,w_i)\}_{i\in I}$ be a fusion frame  and $V=\{(V_i,v_i)\}_{i\in I}$ a Bessel fusion sequence. The \emph{alternate cross-fusion frame operator}   $L_{VW}:\Hil\rightarrow \Hil$ is defined by
$$L_{VW} = T_V\phi_{VW}T_{W}^*.$$
We denote $L_{WW}$ as $L_W$.
\end{defn}
It follows that $\|L_{VW}\|\leq \frac{\sqrt{B_{W}B_V}}{A_{W}}$ and $$\left(L_{VW}\right)^* =\sum_{i\in I}\pi_{W_i}S_W^{-1}\pi_{V_i}.$$
Obviously, $L_{VW}$ is an invertible operator if and only if $V$ is a pseudo-dual of $W$.   In particular, $L_{VW}=I_{\Hil}$ if and only if $V$ is a dual of $W$.

In the following we summarize the basic properties of $L_W.$
\begin{prop}
Let $W=\{(W_i,w_i)\}_{i\in I}$ be a fusion frame. Then $L_W$ is positive, self-adjoint and invertible operator.
\end{prop}
\begin{proof}
It is easy to see that $\phi_{WW}$ is self-adjoint and so, $L_W$ is self-adjoint. Moreover, for all $f\in \Hil$ we have
\begin{eqnarray*}
\left\langle L_Wf,f\right\rangle&=&\left\langle \sum_{i\in I}{w_{i}^2}\pi_{W_i}S_W
^{-1}\pi_{W_{i}}f,f\right\rangle\\
&=&\sum_{i\in I}\left\langle w_{i} \cdot S_W
^{-1/2}\pi_{W_{i}}f,w_{i} \cdot S_W
^{-1/2}\pi_{W_i}f\right\rangle\\
&=&
 \left(\sum_{i\in I}w_i^2\left\|S_W
^{-1/2}\pi_{W_i} f\right\|^2\right)\\
&\geq& \left\|S_{W}^{1/2}\right\|^{-2}\left\|T_{W}^* f\right\|^2
\geq \frac{A_W}{\left\|S_{W}^{1/2}\right\|^{2}}\|f\|^2,
\end{eqnarray*}
where $A_W$ is  a lower bound of the fusion frame $W$.
This shows that $L_W$ is positive and an invertible operator in $B(\Hil).$
\end{proof}

As an easy consequence of Proposition \ref{riesz mitra}, we state
\begin{cor}\label{Wprime}  Let $W=\{(W_{i}, \omega_i)\}_{i\in I}$ be a fusion Riesz basis. Then
 $L_W=S_{W^{\prime}}$.
\end{cor}

\begin{thm}\label{sec:structinvGram1}
Let $W=\{(W_{i}, \omega_i)\}_{i\in I}$ and $V=\{(V_{i},\upsilon_{i})\}_{i\in I}$ be  fusion frames in $\mathcal{H}$ and $U$ an invertible operator in $B(\mathcal{H})$. Then the following assertions hold.
\begin{enumerate}
 \item[(1)] If $\mathcal{G}_{U,W,W}$ is invertible, then $W$ is a fusion Riesz basis and
$$\mathcal{G}_{U,W,W}^{-1} = \mathcal{G}_{S_{W^{\prime}}^{-1}U^{-1}S_{W^{\prime}}^{-1},W,W}.$$
 \item[(2)] If $\mathcal{G}_{U,V,W}$ is invertible and $V$ is a dual of $W$, then $V$ is a fusion Riesz basis and
$\mathcal{G}_{U,V,W}^{-1} = \mathcal{G}_{U^{-1},V,W}$.
\item[(3)]If $\mathcal{G}_{U,W,V}$ is invertible, then $W$ is a fusion Riesz basis and
$$\mathcal{G}_{U,W,V}^{-1} = \mathcal{G}_{( L_{WV} \,  U \, S_{W^{\prime}})^{-1},W,W},$$
\end{enumerate}

\end{thm}
\begin{proof}
\begin{enumerate}
\item[(1)] By Proposition \ref{INV} $W$ is a fusion Riesz basis. Using Corollary \ref{Wprime} follows that
\begin{eqnarray*}
\mathcal{G}_{U,W,W}\mathcal{G}_{S_{W^{\prime}}^{-1}U^{-1}S_{W^{\prime}}^{-1},W,W}&=&
\phi_{WW}T_{W}^*UT_{W}\phi_{WW}T_{W}^*S_{W^{\prime}}^{-1}U^{-1}S_{W^{\prime}}^{-1}T_{W}\\
&=&\phi_{WW}T_{W}^*S_{W^{\prime}}^{-1}T_{W}=I_{\left(\sum_{i\in I}\bigoplus \widetilde{W}_{i}\right)_{\ell^2}}.
\end{eqnarray*}
With the same we have $\mathcal{G}_{S_{W^{\prime}}^{-1}U^{-1}
S_{W^{\prime}}^{-1},W,W}\mathcal{G}_{U,W,W}=I_{\left(\sum_{i\in I}\bigoplus \widetilde{W}_{i}\right)_{\ell^2}}.$

\item[(2)]  According to (\ref{sec:duality}) we have $T_V\phi_{VW}T_{W}^*=I_{\Hil}$. In addition, the invertibility of $\mathcal{G}_{U,V,W}$ implies that $\phi_{VW}T_{W}^*$ has a right inverse. Using Proposition \ref{4} $V$ is a fusion Riesz basis,  and therefore,
\begin{eqnarray*}
\phi_{VW}T_{W}^*T_V=I_{\left(\sum_{i\in I}\bigoplus V_i\right)_{\ell^2}}.
\end{eqnarray*}
Hence,
\begin{eqnarray*}
\mathcal{G}_{U,V,W}\mathcal{G}_{U^{-1},V,W}&=& \phi_{VW}T_{W}^*UT_{V}\phi_{VW}T_{W}^*U^{-1}T_{V}\\
&=& \phi_{VW}T_{W}^*T_{V}=
I_{\left(\sum_{i\in I}\bigoplus V_{i}\right)_{\ell^2}}.
\end{eqnarray*}
Similarly, $\mathcal{G}_{U^{-1},V,W}\mathcal{G}_{U,V,W}^{-1}= I_{\left(\sum_{i\in I}\bigoplus \widetilde{W}_{i}\right)_{\ell^2}}.$ \\
\item[(3)] It follows from Corollary \ref{Wprime}.
\end{enumerate}
\end{proof}

 Repeating the previous argument and using (\ref{sec:duality}) leads to a characterization for fusion Riesz bases due to
 $U$-fusion cross Gram matrices.

\begin{thm}
Let $V=\{(V_{i},\upsilon_{i})\}_{i\in I}$ be a dual fusion frame of $W=\{(W_{i}, \omega_i)\}_{i\in I}$ in $\mathcal{H}$.  The following are
equivalent:
\begin{enumerate}
\item[(1)] $V$ is a fusion Riesz basis.
\item[(2)] $\mathcal{G}_{V,W}=I_{\left(\sum_{i\in I}\bigoplus V_{i}\right)_{\ell^{2}}}$.
    \item[(3)] $\mathcal{G}_{V,W}$ has a left inverse.
    \end{enumerate}
\end{thm}

Consider $W=\{(W_{i},\omega_{i})\}_{i\in I}$, $V=\{(V_{i},\upsilon_{i})\}_{i\in I}$ and $Z=\{(Z_{i},z_{i})\}_{i\in I}$ as fusion frames  in $\mathcal{H}$ and $U_{1}, U_{2}\in B(\mathcal{H})$, it is obvious to see that
\begin{equation*}
\mathcal{G}_{U_{1},W,V}\mathcal{G}_{U_{2},W,Z}=
    \mathcal{G}_{U_{1}T_{W}\phi_{WZ}T_{Z}^{*}U_{2},W,V}.
\end{equation*}
In particular, if $V=\{(V_{i},\upsilon_{i})\}_{i\in I}$ is a dual of $Z=\{(Z_{i},z_{i})\}_{i\in I}$, then
\begin{equation*} \mathcal{G}_{U_{1},V,W}\mathcal{G}_{U_{2},V,Z}=\mathcal{G}_{U_{1}U_{2},V,W}.
\end{equation*}
As a special case, if $U$ is invertible and $V$ a fusion Riesz
basis such that $V$ is a dual of $W$, then $\mathcal{G}_{U,V,W}$ has an inverse in
the form of Gram matrices
\begin{equation*}
\left(\mathcal{G}_{U,V,W}\right)^{-1}=
\mathcal{G}_{U^{-1},V,W}.
\end{equation*}
The above identity is also proved in Theorem \ref{sec:structinvGram1}.
\subsection {Pseudo-inverses}

In the following we discuss the pseudo-inverse of $U$-fusion cross Gram matrices with closed range, and under some conditions we represent their pseudo-inverse as a $U$-fusion cross Gram matrix again motivated by the discrete frame case \cite{mitraG}.
In the following we state a sufficient condition for a $U$-fusion cross Gram matrix having  closed range.

\begin{lem} Let $V=\{(V_{i}, \upsilon_i)\}_{i\in I}$  be a dual  fusion frame of  $W=\{(W_{i}, \omega_i)\}_{i\in I}$    in $\mathcal{H}$ and $U$ an operator in $B(\mathcal{H})$ such that $UV=\{(UV_i, \upsilon_i)\}_{i\in I}$ is also a fusion frame in $\mathcal{H}$. Then $\mathcal{G}_{U,V,W}$ has closed range and
\begin{eqnarray*}
\range{\mathcal{G}_{U,V,W}}&=& \range{\phi_{VW}T_{W}^{*}}.
\end{eqnarray*}
\end{lem}
\begin{proof}
The dual condition (\ref{sec:duality}) follows that  $\phi_{VW}T_{W}^{*}$ has closed range. Thus,
\begin{eqnarray*}
\range{\mathcal{G}_{U,V,W}}&=& \range{\phi_{VW}T_{W}^{*}UT_{V}}\\
&=& \range{\phi_{VW}T_{W}^{*}T_{UV}}\\
&=& \range{\phi_{VW}T_{W}^{*}}.
\end{eqnarray*}
\end{proof}

The assumptions in  the above lemma are  fulfilled, in particular, for  $U$ being invertible \cite{ga07}.

\begin{thm}\label{pse1}
Let $V=\{(V_{i},\upsilon_{i})\}_{i\in I}$ be a dual fusion frame of $W=\{(W_{i}, \omega_i)\}_{i\in I}$ in $\mathcal{H}$, also let $U\in B(\mathcal{H})$ and $\mathcal{G}_{U,V,W}$ have closed range.  The following are equivalent:
 \begin{enumerate}
 \item[(1)]  $\mathcal{G}_{U,V,W}^{\dagger}=\mathcal{G}_{U^{\dagger},V,W}.$
\item[(2)] $\range{\phi_{VW}T_{W}^{*}U^{*}}=\range{T_V^*U^*}$ and $\range{\phi_{VW}T_{W}^{*}U}=\range{T_V^*U}$.
\item[(3)] $\phi_{VW}T_W^*U^*=T_V^*S_V^{-1}U^*$ and $\phi_{VW}T_W^*U=T_V^*S_V^{-1}U$.
 \end{enumerate}
\end{thm}
\begin{proof}
$(2)\Rightarrow (1)$
Applying (\ref{sec:duality}) follows that
\begin{eqnarray*}
\mathcal{G}_{U,V,W}\mathcal{G}_{U^{\dagger},V,W}\mathcal{G}_{U,V,W}&=&\phi_{VW}T_{W}^{*}UT_{V}\phi_{VW}T_{W}^{*}U^{\dagger}T_{V}\phi_{VW}T_{W}^{*}UT_{V}\\
&=&\phi_{VW}T_{W}^{*}UU^{\dagger}UT_{V}\\
&=&\mathcal{G}_{U,V,W}.
\end{eqnarray*}
In addition, (\ref{sec:duality}) is also follows that $T_{W}\phi_{VW}^*$ is surjective. Hence, we have
\begin{eqnarray*}
\range{\mathcal{G}_{U^{\dagger},V,W}}&=& \range{\phi_{VW}T_{W}^{*}U^{\dagger}T_{V}}\\
&=& \range{\phi_{VW}T_{W}^{*}U^{*}}\\
\textrm{(by (2))}&=&\range{T_{V}^*U^*}\\
&=& \range{T_{V}^{*}U^{*}T_{W}\phi_{VW}^*}\\
&=& \range{\mathcal{G}_{U,V,W}^{*}}.
\end{eqnarray*}
Moreover,
\begin{eqnarray*}
\kernel{\mathcal{G}_{U^{\dagger},V,W}}&=& \kernel{\phi_{VW}T_{W}^{*}U^{\dagger}T_{V}}\\
&=& \kernel{U^{\dagger}T_{V}}\\
&=& \range{T_{V}^{*}(U^{\dagger})^{*}}^{\perp}\\
&=& \range{T_{V}^{*}U}^{\perp}\\
\textrm{(by (2))}&=& \range{\phi_{VW}T_{W}^{*}U}^{\perp}\\
&=& \kernel{U^{*}T_{W}\phi_{VW}^{*}}\\
&=& \kernel{\mathcal{G}_{U,V,W}^{*}}.
\end{eqnarray*}
Thus, (1) is obtained. The converse follows immediately from the above identities.
To show
$(2)\Rightarrow (3)$
using Douglas's theorem \cite{doug96}, there is an operator $C\in B(\Hil)$ such that
\begin{eqnarray*}
\phi_{VW}T_W^*U= T_V^*UC.
\end{eqnarray*}
So,
\begin{eqnarray*}
\phi_{VW}T_W^*U&=& T_V^*UC\\
&=&T_V^*S_V^{-1}T_VT_V^*UC\\
&=& T_V^*S_V^{-1}T_V\phi_{VW}T_W^*U=T_V^*S_V^{-1}U.
\end{eqnarray*}
The other identity is obtained similarly. The converse is clear.
\end{proof}

 Note that fusion Riesz bases satisfy in the assumptions of Theorem \ref{pse1}. These assumptions might seem obvious, but are not \cite{ga07}.
In particular, the failure to fulfill this equality makes the concept of duality of fusion frames interesting, so fusion frames are {\em not} 'just
another' generalization of frames.
\begin{rem}
 In the proof of Theorem \ref{pse1}(1) we can replace the fusion frame appeared in the right side by any fusion frame $Z$ such that $V$ is its dual. More precisely, let $W$ and $Z$ be fusion frames and $V$ is a dual of both $W$ and $Z$. Then
$$\mathcal{G}_{U,V,W}^{\dagger}= \mathcal{G}_{U^{\dagger},V,Z},$$
if and only if $\phi_{VZ}T_{Z}^{*}U^{*}=T_V^*S_V^{-1}U^*$ and $\phi_{VZ}T_{Z}^{*}U=T_V^*S_V^{-1}U$.
\end{rem}

Using a similar argument as Theorem \ref{pse1} we obtain the following.
\begin{thm} Let $W=\{(W_{i}, \omega_i)\}_{i\in I}$ be a fusion frame in $\mathcal{H}$, also let $U\in B(\mathcal{H})$ and $\mathcal{G}_{U,V,W}$ have closed range.  The following are equivalent:
 \begin{enumerate}
 \item[(1)]  $(\mathcal{G}_{U,W,W})^{\dagger}= \mathcal{G}_{L_W^{-1}U^{\dagger}L_W^{-1},W,W}.$
\item[(2)] The operators  $\phi_{WW}T_{W}^{*}L_W^{-1}U^{*}$ and $T_W^*U^*$ have the same range as $\phi_{WW}T_{W}^{*}U$ and $T_W^*(L_{W}^{-1})^*U$, respectively, for all $i\in I$
\item[(3)]  $\phi_{WW}T_W^*L_{W}^{-1}U^*=T_W^*S_W^{-1}U^*$ and $\phi_{WW}T_W^*U=T_W^*S_W^{-1}U$.
 \end{enumerate}
\end{thm}

\section{Stability of $U$-cross Gram Matrices of Operators} \label{sec:stabgram0}
In this section, we state a general stability for the invertibility of
$U$-fusion cross Gram matrices, compare to the results on the invertibility of multipliers \cite{balsto09new}.

\begin{thm}\label{stability}
Let $W=\{(W_{i},\upsilon_{i})\}_{i\in I}$ be a Bessel fusion
sequence and $U_{1}$, $U_{2}\in B(\mathcal{H})$ with
$\left\|U_{1}-U_{2}\right\|<\mu$. Also, let
$V=\{(V_{i},\upsilon_{i})\}_{i\in I}$ and
$Z=\{(Z_{i},\upsilon_{i})\}_{i\in I}$ be  fusion frames on
$\mathcal{H}$ such that $\mathcal{G}_{U_{1},W,V}$ is invertible and
\begin{equation}\label{purb}
\left(\sum_{i\in I} \upsilon_i^2 \left\|\pi_{Z_{i}}f-\pi_{V_{i}}f\right\|^{2}\right)^{\frac{1}{2}}\leq
\lambda_{1}\left( \sum_{i\in I} {\upsilon_i^2}  \left\|\pi_{Z_{i}}f\right\|^{2}\right)^{\frac{1}{2}}+\lambda_{2}\left(\sum_{i\in I} {\upsilon_i^2} \left\|\pi_{V_{i}}f\right\|^{2}\right)^{\frac{1}{2}}+\epsilon \|f\|,
\end{equation}
for all $f\in \mathcal{H}$, in addition
\begin{equation}\label{p}
\mu+   \left(\lambda_{1}+\lambda_{2}+\frac{\epsilon}{\sqrt{B}}\right)\left(
\sqrt{B}\left\|S_{Z}^{-1}\right\|\left(\sum_{i\in I} \left|\upsilon_i\right|^{2}\right)^{\frac{1}{2}}\left\|U_{2}\right\|+\left\|U_{2}\right\|
\right)<\frac{\left\|\mathcal{G}_{U_{1},W,V}^{-1}\right\|^{-1}}{B
\left\|S_{V}^{-1}\right\|},
\end{equation}
where $\epsilon>0$,
$B=max\left\{B_{W},B_{V},B_{Z}\right\}$ and $\sum_{i\in I}|\upsilon_i|^2<\infty$. Then
$\mathcal{G}_{U_{2},W,Z}$ is also invertible.
\end{thm}

\begin{proof}
First note that
 \begin{eqnarray*}
 \left\|\left(S_V-S_Z\right)f\right\|&=&\left\|\sum_{i\in I}\upsilon_i^2\left(\pi_{V_i}f-\pi_{Z_i}f\right)\right\|\\
 &\leq&\  \left(\sum_{i\in I}\left|\upsilon_{i}\right|^2\right)^\frac{1}{2}\left(\sum_{i\in I} \upsilon_{i}^2 \left\|\left(\pi_{V_i}f-\pi_{Z_i}f\right)\right\|^2\right)^\frac{1}{2},
 \end{eqnarray*}
 for all $f\in \mathcal{H}$. Therefore,
 \begin{eqnarray*}
&&\left\|\left(S_V^{-1}-S_Z^{-1}\right)f\right\|\\
&=&\left\|S_V^{-1}\left(S_V-S_Z\right)S_Z^{-1}f\right\|\\
&\leq&\left\|S_{V}^{-1}\right\|\left\|\left(S_V-S_Z\right)S_Z^{-1}f\right\|\\
&\leq&\left\|S_{V}^{-1}\right\|\left(\sum_{i\in I} \left|\upsilon_{i}\right|^2\right)^\frac{1}{2}\left(\sum_{i\in I} \left\|\upsilon_{i}^2 \left(\pi_{V_i}-\pi_{Z_i}\right)S^{-1}_Zf\right\|^2\right)^\frac{1}{2}\\
&\leq&\left\|S_{V}^{-1}\right\|\left(\sum_{i\in I}\left|\upsilon_{i}\right|^2\right)^\frac{1}{2}\\
&&\left(\lambda_{1}\left(\sum_{i\in I} \upsilon_{i}^2 \left\|\pi_{Z_i}S_Z^{-1}f\right\|^2\right)^
\frac{1}{2}
+\lambda_{2}\left(\sum_{i\in I} \upsilon_{i}^2  \left\|\pi_{V_i}S_Z^{-1}f\right\|^2\right)^
\frac{1}{2}+\epsilon\left\|S_Z^{-1}f\right\|\right)\\
&\leq&\left\|S_V^{-1}\right\|\left(\sum_{i\in I}\left|\upsilon_{i}\right|^2\right)^\frac{1}{2}\left(\lambda_1\sqrt{B_Z}+\lambda_{2}\sqrt{B_{V}}+\epsilon\right)\left\|S_{Z}^{-1}\right\|\|f\|.
 \end{eqnarray*}

 Using the assumption (\ref{purb}) and the above computations imply that
 \begin{eqnarray*}
&&\left(\sum_{i\in I} \upsilon_i^2 \left\|(S_{V}^{-1}\pi_{V_i}- S_{Z}^{-1}\pi_{Z_i})f\right\|^2\right)^{\frac{1}{2}}\\
&\leq&\left(\sum_{i\in I} \upsilon_i^2 \left(\left\|(S_{V}^{-1}\pi_{V_i}- S_{V}^{-1}\pi_{Z_i})f\right\|+\left\|(S_{V}^{-1}\pi_{Z_i}- S_{Z}^{-1}\pi_{Z_i})f\right\|\right)^{2}\right)^{\frac{1}{2}}\\
&\leq&\left(\sum_{i\in I} \upsilon_i^2 \left(\left\|S_{V}^{-1}\right\|\left\|\pi_{V_i}f- \pi_{Z_i}f\right\|+\left\|S_{V}^{-1}- S_{Z}^{-1}\right\|\|\pi_{Z_i}f\|\right)^{2}\right)^{\frac{1}{2}} \\
&\leq& {  \left[\left(\sum_{i\in I} \upsilon_i^2 \left\|S_{V}^{-1}\right\|^2\left\|\pi_{V_i}f- \pi_{Z_i}f\right\|^2\right)^{\frac{1}{2}}+\left\|S_{V}^{-1}- S_{Z}^{-1}\right\|\left(\sum_{i\in I} \upsilon_i^2\left\|\pi_{Z_i}f\right\|^2\right)^{\frac{1}{2}}  \right]}\\
 \end{eqnarray*}

\begin{eqnarray*}
&\leq&  \left[  \left\|S_V^{-1}\right\|\left(\lambda_1\sqrt{B_Z}+\lambda_{2}\sqrt{B_{V}}+\epsilon\right)\|f\|+ \right. \\
&& \left. \left\|S_V^{-1}\right\|\left\|S_Z^{-1}\right\|\sqrt{B_{Z}}\left(\sum_{i\in I}\left|\upsilon_{i}\right|^2\right)^\frac{1}{2}\left(\lambda_1\sqrt{B_Z}+\lambda_{2}\sqrt{B_{V}}+\epsilon\right)\|f\| \right]\\
&\leq&  \left\|S_V^{-1}\right\|\left(\lambda_1\sqrt{B_Z}+\lambda_{2}\sqrt{B_V}+\epsilon\right)\left(1+\left\|S_Z^{-1}\right\|\sqrt{B_{Z}}\left(\sum_{i\in
I}\left|\upsilon_{i}\right|^2\right)^\frac{1}{2}\right)\|f\|.
 \end{eqnarray*}
 Finally, applying (\ref{p}) we obtain
 \begin{eqnarray*}
&&\left\|\mathcal{G}_{U_{1},W,V}-\mathcal{G}_{U_{2},W,Z}\right\|\\
&=&\left\|\phi_{WV}T_{V}^{*}U_{1}T_{W}-\phi_{WZ}T_{Z}^{*}U_{2}T_{W}\right\|\\
&\leq&\left\|\phi_{WV}T_{V}^{*}U_{1}T_{W}-\phi_{WV}T_{V}^{*}U_{2}T_{W}\right\|+\left\|
\phi_{WV}T_{V}^{*}U_{2}T_{W}-\phi_{WZ}T_{Z}^{*}U_{2}T_{W}\right\|\\
&\leq&\left\|S_{V}^{-1}\right\|\sqrt{B_{W}B_{V}}\left\|U_{1}-U_{2}\right\|+
\left(\sum_{i\in I}\left\| \upsilon_i \pi_{W_i}(S_{V}^{-1}\pi_{V_i}-
S_{Z}^{-1}\pi_{Z_i})U_2T_W
\right\|^2\right)^{\frac{1}{2}}\\
&\leq&\left\|S_{V}^{-1}\right\|\sqrt{B_{W}B_{V}}\mu+ \left[ \sum_{i\in I} \upsilon_i^2 \left\|(S_{V}^{-1}\pi_{V_i}- S_{Z}^{-1}\pi_{Z_i}) \right\|^2 \right]^{\frac{1}{2} }\|U_2\|\sqrt{B_{W}}\\
&\leq&\left\|S_{V}^{-1}\right\|\sqrt{B_{W}B_{V}}\mu+  \\
&&\left\|S_V^{-1}\right\|\left(\lambda_1\sqrt{B_Z}+\lambda_{2}\sqrt{B_{V}}+\epsilon\right)\left(1+\left\|S_Z^{-1}\right\|\sqrt{B_{Z}}\left(\sum_{i\in I}\left|\upsilon_{i}\right|^2\right)^\frac{1}{2}\right)\|U_2\|\sqrt{B_{W}}\\
&<&\left\|\mathcal{G}_{U_{1},W,V}^{-1}\right\|^{-1}.
 \end{eqnarray*}
Therefore,  $\mathcal{G}_{U_{2},W,Z}$ is also invertible by Proposition \ref{invOP}.
\end{proof}

\begin{rem} It is worthwhile to mention that if we consider in Theorem  \ref{stability}
\begin{enumerate}
\item  the perturbation condition
\begin{equation*}
\left\|\pi_{Z_{i}}f-\pi_{V_{i}}f\right\| \leq
\lambda_{1}  \upsilon_i \left\|\pi_{Z_{i}}f\right\|+\lambda_{2}{\upsilon_i} \left\|\pi_{V_{i}}f\right\|+\epsilon \|f\|
\end{equation*}
  we can replace the assumption $\sum_{i\in I}|\upsilon_i|^2<\infty$ by bounded weights. Hence,
uniform fusion frames satisfy a slightly different version of this theorem.

\item the condition
\begin{equation*}\label{purb_alt}
\left(\sum_{i\in I} \left\|\pi_{Z_{i}}f-\pi_{V_{i}}f\right\|^{2}\right)^{\frac{1}{2}}\leq
\lambda_{1}\left( \sum_{i\in I} \upsilon_i^2  \left\|\pi_{Z_{i}}f\right\|^{2}\right)^{\frac{1}{2}}+\lambda_{2}\left(\sum_{i\in I}  \upsilon_i^2 \left\|\pi_{V_{i}}f\right\|^{2}\right)^{\frac{1}{2}}+\epsilon \|f\|,
\end{equation*}  instead of (3.4)
we get the same result substituting $\sum_{i\in I}|\upsilon_i|^4$ for $\sum_{i\in I}|\upsilon_i|^2$.

\end{enumerate}
\end{rem}

As a consequence we obtain the following result.
 \begin{cor}
Let $W=\{(W_{i},\upsilon_i)\}_{i\in I}$ be a fusion Riesz basis with bounds
$A_{W}$ and
 $B_{W}$ and $Z=\{(Z_{i},\upsilon_i)\}_{i\in I}$  a  fusion frame in
 $\mathcal{H}$ such that  (\ref{purb}) holds. Also, $U\in B(\mathcal{H})$ with $\left\|U-I_{\mathcal{H}}\right\|<\mu$. If
\begin{equation*}
\mu+\left(\lambda_{1}+\lambda_{2}+\frac{\epsilon}{\sqrt{B}}\right)\left(
\sqrt{B}\left\|S_{Z}^{-1}\right\|\left(\sum_{i\in I}|\upsilon_i|^2\right)^{\frac{1}{2}}\|U\|+\|U\|
\right)<\frac{A_{W}}{B \left\|S_{W}^{-1}\right\|},
\end{equation*}
where  $\epsilon>0$ and
$B=max\left\{B_{W},B_{Z}\right\}$ and $\sum_{i\in I}|\upsilon_i|^2<\infty$. Then $\mathcal{G}_{U,W,Z}$
is also invertible.
\end{cor}

\section{Acknowledgement}
The first and the second authors  were supported in part by the Iranian National Science Foundation (INSF) under Grant 97018155.
The last author was partly supported by the
START project FLAME Y551-N13
of the Austrian Science Fund (FWF) and the DACH project BIOTOP
I-1018-N25 of Austrian Science Fund (FWF). He also thanks Nora Simovich for help with typing.\\

\end{document}